\documentclass[a4paper,10pt]{article}
\usepackage[utf8]{inputenc}
\usepackage{amsmath,amssymb,amsthm} 
\newcommand{\R}{\mathbb{R}}
\newcommand{\N}{\mathbb{N}}
\newcommand{\Lone}{\mathcal{L}_1}
\newcommand{\Ltwo}{\mathcal{L}_2}
\newcommand{\sign}{\text{\rm sign}}

\newtheorem{definition}{Definition}
\newtheorem{proposition}{Proposition}
\newtheorem{corollary}{Corollary}
\newtheorem{lemma}{Lemma}
\newtheorem{theorem}{Theorem}
\newtheorem{remark}{Remark}
\newcommand{\Lams}{\Theta}
\newcommand{\con}{\psi}
\newcommand{\f}{{\theta}}
\newcommand{\h}{{\sigma}}
\newcommand{\LM}[2]{ {L_{#2}^{#1}}} 
\newcommand{\LCc}[2]{ {c_{#2}^{#1}}} 
\newcommand{\scal}{\langle}
\newcommand{\scar}{\rangle}
\newcommand{\cff}{p}
\newcommand{\cgg}{q}
\newcommand{\myg}{\zeta}
\newcommand{\ccc}{\eta}
\newcommand{\ddd}{\xi}
\renewcommand{\div}{\nabla.}

\title{Inversion formulas for the linearized impedance tomography problem}
\author{Stefan Kindermann}

\begin{document}
\maketitle
\begin{abstract}         
We consider the linearized electrical impedance tomography problem 
in two dimensions on the unit disk. By a linearization around constant
coefficients and using a trigonometric basis, we calculate the linearized 
Dirichlet-to-Neumann operator in terms of moments of the conduction
coefficient of the problem. By expanding this coefficient into angular 
trigonometric functions and Legendre-M\"untz polynomials in radial 
coordinates, we can find a lower-triangular representation of the 
parameter to data mapping. As a consequence, we find an explicit 
solution formula for the corresponding inverse problem.  Furthermore,
we also consider the problem with boundary data given only 
on parts of the boundary while setting homogeneous Dirichlet values on the rest. 
We show that 
the conduction coefficient is uniquely determined from incomplete 
data of the linearized Dirichlet-to-Neumann operator with an explicit 
solution formula provided.
\end{abstract}

\section{Introduction}
A classical parameter identification problem is to reconstruct 
certain parameter function in a second-order elliptic equation 
from multiple measurements of the boundary values and 
boundary fluxes of the solution. 

Specifically, we consider in this article two types of  elliptic equations on
the unit ball in $\R^2$. Define the differential operators  
\[ \Lone u = -\div(\gamma \nabla) u \qquad \qquad \Ltwo u = -\Delta + c u,  \]
where $\div$ denotes divergence, $\nabla$ is the gradient, and 
 $\gamma >0$ and $c$ are sufficiently regular functions. 

We may associate to each operator the solutions to the Dirichlet problem on the unit ball
\begin{equation}\label{main1}
\begin{split}
\Lone u_f & = 0 \qquad \mbox{ in } \Omega  \\
 u_f & = f \qquad \mbox{ on } \partial \Omega, 
\end{split}
 \end{equation}
or in the second case, 
\begin{equation}\label{main2}
\begin{split}
\Ltwo u_f & = 0 \qquad \mbox{ in } \Omega  \\
 u_f & = f \qquad \mbox{ on } \partial \Omega, 
 \end{split}
 \end{equation}
 where $\Omega = \{ (x,y) \in \R^2 \,|\, x^2 +y^2 < 1 \}$ and 
$f$ is  a sufficiently regular function. 

Under well-known conditions on $\gamma$ or $c$ and $f$, these problems
have a solution in $H^1(\Omega)$. For instance, for $f \in H^{1/2}(\partial \Omega)$ and
$c_1\leq \gamma \leq c_0$ 
almost everywhere,  a solution in $H^1$ to \eqref{main2} exists. Also, if, 
$f \in H^{1/2}(\partial \Omega)$ and,
e.g.,  
$0 \leq c \leq c_1$, then 
a solution to \eqref{main2} exists; see, e.g., \cite{Is}.

For fixed parameter functions $\gamma$ (respectively $c$), we may consider 
the Dirichlet-to-Neumann mapping, 
\[ \Lambda_{1,\gamma}: f \to  \gamma \frac{\partial}{\partial n} u_f  \qquad  \qquad
 \Lambda_{2,c}: f \to  \frac{\partial}{\partial n} u_f, \]
where $u_f$ is the solution to \eqref{main1} for $\Lambda_1$ and 
\eqref{main2} for $\Lambda_2$, respectively.  Under mild conditions, 
these mappings are continuous from $H^{1/2}(\partial \Omega)$ to 
 $H^{-1/2}(\partial \Omega)$.

The classical inverse problem of electrical impedance tomography, originating in the 
famous paper by Calder\'on \cite{Cal}, asks to 
reconstruct the parameter function $\gamma$ from  knowledge of the mapping $\Lambda_{1,\gamma}$, i.e., 
from all pairs of $(f,\gamma \frac{\partial}{\partial n} u_f)$ of Dirichlet and 
Neumann values.  
The similar problem has been stated also for the case of $\Ltwo$  (i.e., the Schrödinger equation),
where $c$ is sought to be found from $\Lambda_{2,c}$. 

Both are well-studied and classical inverse problems for partial differential 
equation; the central difficulty lies in the fact that only the boundary is accessible 
for measurements, while $\gamma$ is sought to be reconstructed in the interior. 
Concerning the unique identifiability of $\gamma$ (or $c$)  from 
the Dirichlet-to-Neumann map, several landmark papers were published, 
for instance, by Sylvester and Uhlmann \cite{SylUl1,SylUl2}, 
Nachman \cite{Nach}, and Astala and P\"aiv\"arinta \cite{AsPa}. An overview of results
and related problems can be found in the classical book by Isakov \cite{Is} as well 
as in the review articles \cite{Borcea,Uhlmanrev}.

In this article we consider only the linearized versions of these problems, namely, 
to find $\gamma$ or $c$ when the Dirichlet-to-Neumann maps are linearized 
(with respect to $\gamma$, $c$) 
around a constant. It will be shown, amongst others, that one can find explicit reconstruction 
formulas in these cases.

\section{Problem setup} 
Considering the problems \eqref{main1} and \eqref{main2}, 
it is well-known that certain differences of  Dirichlet-to-Neumann maps can be expressed as 
energy integrals: for $f, g \in H^{1/2}(\partial \Omega)$, it holds that \cite[Eq. (5.0.3)]{Is} 
\begin{align} 
\langle (\Lambda_{1,\gamma+1} - \Lambda_{1,1})f, g\rangle_{H^{-1/2},H^{1/2}}
& = \int_{\Omega} \gamma(x) \nabla u_{f,\gamma+1}(x). \nabla u_{g,1}(x) dx,  \label{int1} \\
\langle (\Lambda_{2,c} - \Lambda_{2,0})f, g\rangle_{H^{-1/2},H^{1/2}}
& =  \int_{\Omega} c(x) u_{f,c}(x) u_{g,0}(x) dx, \label{int2} 
\end{align}
where $ u_{f,\gamma+1}$  is  the solution to \eqref{main1} with coefficient $\gamma+1$ and 
Dirichlet data $f$, and  $u_{g,1}$  is the solution to \eqref{main1}  with coefficient 1 
(i.e., the Laplace equation) and 
Dirichlet data $g$.  Similarly  $u_{f,c}, u_{g,0}$ are the solutions to \eqref{main2} with coefficient 
$c$ and Dirichlet values $f$ 
and $0$ (Laplace equation) and Dirichlet values $g$, respectively.

Note that the right hand side in \eqref{int1}--\eqref{int2} depends in a nonlinear way on the parameter through 
the functions $u_{f,\gamma+1}$ and $u_{f,c}$. Thus, in a next step, we consider a 
linearization with respect to $\gamma$ or $c$ of the 
right-hand side around the constant $\gamma= 1$ and $c = 0$.  We thus view only 
small/moderate perturbation of $\gamma$, respectively $c$,  around a constant conductivity to be of interest. 
This yields, the linearized impedance tomography problem with the 
following operators. 
\begin{align} 
\langle {\Lambda'}_{1}f, g\rangle_{H^{-1/2},H^{1/2}}
& = \int_{\Omega} \gamma(x) \nabla u_{f,1}. \nabla u_{g,1} dx,  \label{intlin1} \\
\langle {\Lambda'}_{2}f, g\rangle_{H^{-1/2},H^{1/2}}
& =  \int_{\Omega} c(x) u_{f,0} u_{g,0} dx. \label{intlin2} 
\end{align}

It is well-known that under mild conditions, these operators correspond to the linearization 
of the associated parameter-to-data mappings. In fact, e.g., for $\gamma,c \in L^\infty(\Omega)$, 
one can verify that this 
is indeed the output of the Fr\'{e}chet-derivative of these mappings. 

\begin{definition}
The linearized impedance tomography problem for  \eqref{main1} is the problem to find  the
function $\gamma$ from the values 
$\{ \langle {\Lambda'}_{1} f,g\rangle\,|\,  \text{ for all } f,g \in H^{1/2}(\partial \Omega) \}$. 
The linearized  tomography problem for  \eqref{main2} is the problem to find the function $c$ from the values 
$\{ \langle {\Lambda'}_{2} f,g\rangle\,|\,  \text{ for all } f,g \in H^{1/2}(\partial \Omega) \}$. 
\end{definition} 

Of course, it is enough, to known $\langle {\Lambda'}_{1} f,g\rangle$ for all $f,g$ out of a 
basis of $H^{1/2}(\partial \Omega)$. 
In the next section, we find a formula for these operators in the trigonometric basis. 

\section{Linearized Dirichlet-to-Neumann maps in the trigonometric basis} 
Specifically, we now consider  ${\Lambda'}_{1}, {\Lambda'}_{2}$, when these operator are applied to 
trigonometric functions.  That is, we consider the family of functions
(living on the boundary of the unit disc)
$$ \{\cos(n \phi)\,|\,n \in \N_0 , n\geq 0\} \cup \{\sin(n \phi)\,|\,n \in \N , n\geq 1\}, \quad \phi \in [0, 2 \pi]. $$ 
These functions consist a basis of the space $H^{1/2}(\partial \Omega)$.
We note that for the impedance tomography problem, we do not have to include $\cos(n \phi)$ for $n = 0$, i.e., 
the constant function, because it is in the nullspace of the Dirichlet-to-Neumann map and hence does not
provide any information.

If $f = \cos(n \phi)$ or $f = \sin(n \phi)$, then the corresponding solutions to \eqref{main1} and \eqref{main2}
for the parameter 
$\gamma = 1$ or $c = 0$ are given in polar coordinates as
\begin{equation}\label{sols} u_f(r,\phi) = r^n \cos( n \phi) \qquad \mbox{ or } \qquad  u_f(r,\phi) = r^n \sin( n \phi), \end{equation}
respectively, where $r = \sqrt{x^2+y^2}$, $\phi = \arctan(y/x)$. 

\begin{definition}\label{def2}
Define the coefficients of the linearized Dirichlet-to-Neumann operators 
 ${\Lambda'}_{1}$,  ${\Lambda'}_{2}$
in the trigonometric basis as follows: 
\begin{align*}
K^{cc}_{i,j}  &:= \langle {\Lambda'}_{1} \cos(i \phi) ,\cos(j \phi) \rangle&   &i,j \in \N,& \\ 
K^{ss}_{i,j}  &:= \langle {\Lambda'}_{1} \sin(i \phi) ,\sin(j \phi) \rangle&  &i,j \in \N,& \\ 
K^{sc}_{i,j}  &:= \langle {\Lambda'}_{1} \sin(i \phi) ,\cos(j \phi) \rangle&  &i,j \in \N,& \\ 
K^{cs}_{i,j}  &:= \langle {\Lambda'}_{1} \cos(i \phi) ,\sin(j \phi) \rangle&   &i,j \in \N& 
\end{align*}
and 
\begin{align*}
J^{cc}_{i,j}  &:= \langle {\Lambda'}_{2} \cos(i \phi) ,\cos(j \phi) \rangle&  &i,j \in \N_0,& \\ 
J^{ss}_{i,j}  &:= \langle {\Lambda'}_{2} \sin(i \phi) ,\sin(j \phi) \rangle&  &i,j \in \N,& \\ 
J^{sc}_{i,j}  &:= \langle {\Lambda'}_{2} \sin(i \phi) ,\cos(j \phi) \rangle& &i \in \N,j\in \N_0,& \\ 
J^{cs}_{i,j}  &:= \langle {\Lambda'}_{2} \cos(i \phi) ,\sin(j \phi) \rangle& &j \in \N, i \in \N_0.& 
\end{align*}
\end{definition} 

Similarly, we may represent $\gamma$ and $c$ in polar coordinates as a Fourier series 
with respect to the angle coordinates:
\begin{align} 
 \gamma(r,\phi) &= a_0 + \sum_{n = 1}^\infty  a_n(r) \cos(n \phi) + b_n (r) \sin( n \phi),   \label{par1xx} \\
c(r,\phi) &= a_0 + \sum_{n = 1}^\infty     a_n(r) \cos(n \phi) + b_n (\phi) \sin( n \phi).   \label{par2}
\end{align}
In order for $\gamma,c$ to be  $L^2(\Omega)$ functions, it is necessary and sufficient that 
\begin{equation}\label{eltwo}
\sum_{n=0}^\infty \int a_n(r)^2 r dr + \sum_{n=1}^\infty \int b_n(r)^2 r dr  < \infty. 
\end{equation}

We can now express the parameter-to-data operator in the trigonometric basis. 
\begin{proposition}
Let $\gamma$ be given by \eqref{par1xx}.  Then 
\begin{align*}  K^{cc}_{i,j}  = K^{ss}_{i,j}
& = i j \pi \myg_{i,j}  \int_0^1 r^{i+j-1} a_{|i-j|}(r) dr & \quad & i,j\geq 1,&  
\\ 
K^{cs}_{i,j}  = K^{sc}_{j,i} & =
  i j \pi  \sign(j-i)   \int_0^1 r^{i+j-1} b_{|i-j|}(r) dr    & \quad & i,j\geq 1,&  
\end{align*}
where 
%
%
\[ \myg_{i,j} = \begin{cases} 1 & |i-j|\geq 1, \\ 
2 & |i-j| = 0, \end{cases}  \]
and $\sign$ is the sign function with $\sign(0) = 0$.
\end{proposition} 
\begin{proof}
It can be verified that for $u = r^i \cos( i\phi)$ and $v = r^j \cos(j \phi)$ and for 
$u = r^i \sin( i\phi)$ and $v = r^j \sin(j \phi)$ given in polar coordinates, we have 
\[  \nabla u.\nabla v = i j r^{i+j-2} \cos((i-j)\phi). \]
From 
\[ \int_0^{2\pi}  \cos((i-j)\phi) \cos(k \phi) d\phi = \begin{cases} \pi  & k = |i-j| \& |i-j| \not = 0, \\
2 \pi  & k = |i-j| \& |i-j| =0, \\
0 & \text{else}, \end{cases} \]
\eqref{intlin1}, 
and an integration in polar coordinates gives the identity for $K^{cc}$ and $K^{ss}$. 
(Note that the terms $\sin(k \phi)$ in $\gamma$ do not contribute because of the identity
\mbox{$ \int_0^{2\pi}  \cos((i-j)\phi) \sin(k \phi) d\phi = 0$}. 
For $u = r^i \cos( i\phi)$ and $v = r^j \sin(j \phi)$, we have
\[  \nabla u.\nabla v = i j r^{i+j-2} \sin((j-i)\phi), \]
and 
\[ \int_0^{2\pi}  \sin((j-i)\phi) \sin(k \phi) d\phi = \begin{cases} \pi  & k = j-i \& |i-j| \not = 0, \\
- \pi  & k = i-j \& |i-j| \not=0, \\
0 & \text{else}, \end{cases} \]
and  in this case the cosine terms in $\gamma$ cancel after integration. 
\end{proof}
\begin{proposition}
Let $c$ be given by \eqref{par2}. Then
\begin{align*}  J^{cc}_{i,j} &=   \frac{\pi}{2} \int_0^1 a_{i+j}(r)  r^{i+j+1} dr    +   
\ccc_{i,j} \frac{\pi}{2} \int_0^1 a_{|i-j|}(r)  r^{i+j+1} dr& & i,j\geq 0,  \\
J^{ss}_{i,j} &=   -  \frac{\pi}{2} \int_0^1 a_{i+j}(r)  r^{i+j+1} dr  +   
\ddd_{i,j} \frac{\pi}{2} \int_0^1 a_{|i-j|}(r)  r^{i+j+1} dr& & i,j\geq 1,   \\
J^{sc}_{i,j} &=      \frac{\pi}{2}   \int_0^1 b_{i+j}(r)  r^{i+j+1} dr +  \sign(i-j) 
\frac{\pi}{2}   \int_0^1 b_{|i-j|}(r)  r^{i+j+1} dr&  &i\geq 1,j\geq 0,&\\
J^{cs}_{i,j} &=  
     \frac{\pi}{2}   \int_0^1 b_{i+j}(r)  r^{i+j+1} dr   -\sign(i-j) \frac{\pi}{2}   \int_0^1 b_{|i-j|}(r)  
     r^{i+j+1} dr&  &i\geq 0, j \geq 1,& 
\end{align*} 
where 
\[ \ccc_{i,j} = \begin{cases} 1 &  |i-j|\geq 1,  \\
2  & |i-j| = 0 \& (i,j) \not = (0,0), \\
3  &   (i,j) = (0,0), \end{cases} 
\qquad 
\ddd_{i,j} = \begin{cases} 1 &  |i-j|\geq 1 , \\
2  & |i-j| = 0 \& (i,j) \not = (0,0). 
\end{cases} 
\] 
\end{proposition} 

\begin{proof}  The proof is based on the following integral that follow from trigonometric identities and  orthogonality: 
for $i,j,k \in \N$ and $i,j,k \geq 0$, 
\[  \int_0^{2\pi}  \cos(k \phi)  \cos(i \phi) \cos( j \phi)  d\phi  = \frac{\pi}{2} \left( \delta_{0,k+i+j} +  
 \delta_{0,k-i+j} +  \delta_{0,k+i-j} +  \delta_{0,k-i-j}   \right),
 \] 
\[  \int_0^{2\pi}  \cos(k \phi)  \sin(i \phi) \sin( j \phi) d\phi  = \frac{\pi}{2} \left( \delta_{0,k+i-j} +  
 \delta_{0,k-i+j} -  \delta_{0,k+i+j} -  \delta_{0,k-i-j}   \right),
 \]  
and 
\[  \int_0^{2\pi}  \sin(k \phi)  \cos(i \phi) \cos( j \phi)  = 0, \qquad 
 \int_0^{2\pi}  \sin(k \phi)  \sin(i \phi) \sin( j \phi) = 0,\]
where $\delta_{i,j}$ denotes the Kronecker delta. 
Thus, denote the zero-extension of the coefficients $a_k$ to negative indices $k$ by 
$\hat{a}_k$, we have  
\begin{align*}
 J^{cc}_{i,j} &=    \frac{\pi}{2}   \sum_{k=0}^\infty \int_0^1 r^{i+j+1} a_k(r) dr 
 \left( \delta_{0,k+i+j} +  
 \delta_{0,k-i+j} +  \delta_{0,k+i-j} +  \delta_{0,k-i-j}   \right) \\
 &= \frac{\pi}{2}  \int_0^1 r^{i+j+1} \hat{a}_{-i-j}(r)  +  \hat{a}_{i-j}(r) +  \hat{a}_{-i+j}(r) dr  + 
   \hat{a}_{i+j}(r) dr  \\
   & =  \frac{\pi}{2}   \int_0^1 r^{i+j+1}    a_{i+j}(r) dr  + 
\begin{cases}     \frac{\pi}{2}   \int_0^1 r^{i+j+1}    a_{|i-j| }(r) dr  & |i-j|\geq 1 \& i,j\geq 1, \\
  2 \frac{\pi}{2}   \int_0^1 r^{i+j+1}    a_{|i-j| }(r) dr  & |i-j|=0 \&  (i,j) \not = (0,0), \\ 
     3 \frac{\pi}{2}   \int_0^1 r^{i+j+1}    a_{|i-j| }(r) dr  & |i-j|=0 \& i=j= 0, 
     \end{cases}
 \end{align*}
and for $i,j\geq 0$,
\begin{align*}
 J^{ss}_{i,j} & = 
   \frac{\pi}{2}   \sum_{k=0}^\infty \int_0^1  r^{i+j+1} a_k(r) dr  
   \left( \delta_{0,k+i-j} +  
 \delta_{0,k-i+j} -  \delta_{0,k+i+j} -  \delta_{0,k-i-j}   \right) \\ 
 & =  \frac{\pi}{2} \int_0^1  r^{i+j+1} \left( \hat{a}_{j-i}(r) +  \hat{a}_{i-j}(r) -
 \hat{a}_{-j-i}(r) - \hat{a}_{j+i}(r)  \right) dr  \\
  & = -  \frac{\pi}{2} \int_0^1  r^{i+j+1}  a_{j+i}(r)  
  + \begin{cases}  \frac{\pi}{2}   \int_0^1 r^{i+j+1}    a_{|i-j|}(r) dr  
  & |i-j|\geq 1 \& i,j\geq 1, \\
   2 \frac{\pi}{2}   \int_0^1 r^{i+j+1}    a_{|i-j| }(r) dr
  & |i-j|=0 \&  i,j \geq 1. \\ 
  \end{cases}
 \end{align*}
 Moreover, for $i\geq1$ and $j\geq 0$, 
 \begin{align*}
 J^{sc}_{i,j} & = 
    \frac{\pi}{2}   \sum_{k=0}^\infty \int_0^1  r^{i+j+1} b_k(r) dr 
    \left( \delta_{0,j+i-k} +  
 \delta_{0,j-i+k} -  \delta_{0,j+i+k} -  \delta_{0,j-i-k}   \right)  \\
 & = 
     \frac{\pi}{2} 
      \int_0^1  r^{i+j+1} \left(\hat{b}_{j+i}(r)  +   \hat{b}_{i-j}(r)  
      -\hat{b}_{-j-i} - \hat{b}_{j-i}\right) \\
& =      \frac{\pi}{2}       \int_0^1  r^{i+j+1} b_{j+i}(r) dr 
+ \sign(i-j) \frac{\pi}{2}   \int_0^1  r^{i+j+1} b_{|i-j|}(r) dr. 
\end{align*}
\end{proof}

As a consequence, we can characterize what ``algebraic'' condition the 
linearized Dirichlet-to-Neumann map in the trigonometric basis has to satisfy.
We have some trivial conditions that arise from the symmetry of the Dirichlet-to-Neumann map, namely, 
\begin{equation}\label{thisy}
\begin{split}
&K^{cc}, K^{ss} \text{ are symmetric,  } \quad K^{cs} = {K^{sc}}^T,   \\
&J^{cc}, J^{ss} \text{ are symmetric,  and } J^{cs} = {J^{sc}}^T. 
\end{split}
\end{equation} 
Besides this, we have the following nontrivial conditions.
\begin{proposition}
Let $K^{cc},K^{ss}, K^{cs},K^{sc}$ and  $J^{cc},J^{ss}, J^{cs},J^{sc}$ be as above. 
Then, besides of \eqref{thisy} we have that 
\begin{equation}\label{misy} K^{cs}   = -{K^{cs}}^T  \quad  K^{sc}   = -{K^{sc}}^T, 
\end{equation}
i.e., they are antisymmetric. 

In the Schrödinger case, we have that 
\[ J^{sc} - J^{cs}  \text{ is antisymmetric }, \]
and 
\[ (J^{ss} - J^{cc})_{i,j \geq 0} \text{ and }   (J^{sc} + J^{cs})_{i,j \geq 0} \text{ are Hankel matrices}. \]
\end{proposition}

Note that Hankel matrices have entries $A_{i,j}$ that only depend on $i+j$; in the formulas 
above, the entries  in the Hankel matrices involve 
the terms  $\int_0^1 r^{i+j+1}    a_{i+j}(r) dr$ and 
$\int_0^1  r^{i+j+1} b_{j+i}(r) dr$, respectively.

By rearranging the entries in the Dirichlet-to-Neumann map, the identification problem can be rephrased differently:
Define
\begin{equation}  d_i^{c,k}:= \frac{1}{i (i+k) \pi} K_{i,i+k}^{cc} \quad 
 d_i^{s,k}:= \frac{1}{i (i+k) \pi} K_{i,i+k}^{cs},  \qquad  i = 1,\ldots, \quad k = 0,\ldots \end{equation}
then, in order  to find $\gamma$ in the form 
\begin{align} 
 \gamma(r,\phi) &=   \frac{a_0}{2} + \sum_{n = 1}^\infty  a_n(r) \cos(n \phi) + 
 b_n (r) \sin( n \phi),   \label{par1} 
\end{align}
we have to solve  the moment problems 
\begin{equation}\label{momprob}
\begin{split}
    \int_0^1     r^{2 i + k -1} a_k(r) dr   &= d_i^{c,k} \quad i = 1,\ldots, \quad  k = 0,\ldots \\ 
    \int_0^1     r^{2 i + k -1} b_k(r) dr   &= d_i^{s,k} \quad i = 1,\ldots, \quad k = 0,\ldots 
\end{split}
\end{equation}
Similarly,  in the Schrödinger case, we define 
\begin{align*}
d_i^{c,0}&:= \begin{cases}  \frac{1}{\pi } J_{0,0}^{cc} & i = 0, \\
\frac{1}{\pi} (J^{cc} + J^{ss})_{i,i}        &   i \geq 1, \end{cases} \\
d_i^{c,k} &:=  \begin{cases}  \frac{1}{\pi} (J^{cc} - J^{ss})_{k,0} & i = 0, \\
\frac{1}{\pi} (J^{cc} + J^{ss})_{i,i+k} & i \geq 1,  \end{cases}  \quad k \geq 1, \\ 
d_i^{s,k}  &:=\begin{cases} \frac{1}{\pi} (J^{cs}_ {0,k} + J^{sc}_{k,0})  & i = 0, \\
\frac{1}{\pi} (J^{cs} - J^{sc})_{i,i+k}  &  i \geq 1,  \end{cases}   \quad k \geq 1. 
\end{align*}
Then, to find $c$ 
in the form 
\begin{align} 
 c(r,\phi)&=   \frac{a_0}{2} + \sum_{n = 1}^\infty 
 a_n(r) \cos(n \phi) + b_n (r) \sin( n \phi),   \label{par1c} 
\end{align}
we have to solve the moment problems 
\begin{equation}\label{oned}
\begin{split} 
    \int_0^1     r^{2 i + k+1} a_k(r) dr   &= d_i^{c,k} \quad i = 0,\ldots, \quad k = 0,\ldots   \\ 
    \int_0^1     r^{2 i + k +1} b_k(r) dr   &= d_i^{c,k} \quad i = 0,\ldots, \quad k = 0,\ldots, 
\end{split}
\end{equation}
where, by \eqref{eltwo}, the coefficients $a_k,b_k$ are sought
such that  $\sqrt{r} a_k(r),  \sqrt{r} b_k(r),$ are in $L^2([0,1])$. 

Thus, up to a shift in the index $i$, both parameter identification problems 
lead to the same moment problems.  In the next section, we study in detail their solution 
by Müntz-Legendre polynomials. 

\section{Inversion Formula}
We recall the definition of the Müntz-Legendre polynomials (see, e.g., \cite{Borwein}): 
\begin{definition}\label{def3}
For a sequence   of real numbers with disjoint elements, $\Lams:= (\lambda_i)_i$, $ \lambda_i \geq -\frac{1}{2}$, $i = 0,\ldots$, the Müntz-Legendre polynomials are defined as
\begin{equation}  L_n(x)  = \sum_{k=0}^n c_{k,n} x^{\lambda_k} \qquad c_{k,n} = \frac{\Pi_{j=0}^{n-1} (\lambda_k + \lambda_j + 1)}
{\Pi_{j=0,j\not = k}^{n} (\lambda_k - \lambda_j) },\quad n = 0,\ldots. 
\end{equation}
\end{definition} 
Here the product over an empty set of indices is by definition $1$.
It is well-known that the functions $(L_n)_n$ are orthogonal  \cite[Theorem 2.4]{Borwein}(but not normalized) with respect to the 
$L^2([0,1])$-inner product. These polynomials are named after C. Müntz, who proved the 
famous result \cite{muntz} that the powers $(x^\lambda_i)_i$ 
 span the space 
$L^2([0,1])$ if and only 
if the series  $\sum_{i=0,\lambda_i\not=0}^\infty  \frac{1}{\lambda_i}$ diverges.  

The functions $L_n$ are up to a normalization constant identical to the result of a Gram-Schmidt procedure 
applied to the monomials $ x^{\lambda_i}$.  In particular, it follows that 
\[  \scal L_n(x), \text{span}\{x^{\lambda_i}\,|\, i < n \}\scar_{L^2([0,1])} = 0, \]  
and it is easy to verify (e.g., by induction) that 
\[ \text{span} \{ L_n(x)\, | 0\leq n \leq N\} =  \text{span}\{ x^{\lambda_n}\,| 0 \leq n \leq N \}. \] 

The coefficients of the monomials $x^{\lambda_i}$ with respect to the Müntz-Legendre-polynomials can be
explicitly calculated:

\begin{lemma}\label{lemone}
For a sequence $\Lams = (\lambda_i)_i$ as in Definition~\ref{def3}, let 
\begin{equation}\label{myA} 
A_{l,n}:= \scal L_n(x), x^{\lambda_l}\scar_{L^2([0,1])} \qquad \text{ for } l,n = 0,\ldots \end{equation}
Then
\[ A_{l,n} =  \frac{ \Pi_{j=0}^{n-1} (\lambda_l -\lambda_{j})}{\Pi_{j=0}^{n}(1 + \lambda_l + \lambda_j)},\] 
in particular $A_{l,n} = 0$ for $n > l.$
\end{lemma}

\begin{proof}
From the orthogonality it follows that $A_{l,n} = 0$ for $n>l$. Thus, consider the case of $n \leq l$. 
The Müntz-Legendre polynomials  satisfy the recurrence  \cite{Borwein}
\[  x L_{n}'(x) - x L_{n-1}'(x) = \lambda_n L_n(x) + (1+ \lambda_{n-1}) L_{n-1}(x).  \]
Multiply this identity by $x^{\lambda_l}$ and integrate by parts to obtain 
\begin{align*} 
&- (\lambda_l+1) \int_0^1 x^{\lambda_l} L_n(x) dx + x^{\lambda_l +1} L_n(x)|_0^1 \\
&- [ -(\lambda_l+1) \int_0^1 x^{\lambda_l} L_{n-1}(x) dx +  x^{\lambda_l +1} L_{n-1}(x)|_0^1]  \\
&= 
\lambda_n A_{l,n} + (1+ \lambda_{n-1}) A_{l,n-1}. 
\end{align*}
We have that $L_n(1) = 1$, \cite{Borwein}, thus, 
\[  - (\lambda_l+1)  A_{l,n} + (\lambda_l+1) A_{l,n-1} = \lambda_n A_{l,n} + (1+ \lambda_{n-1}) A_{l,n-1}, \]
which gives the recursion 
\[ (1 +  \lambda_l + \lambda_n) A_{l,n} =  (\lambda_l -\lambda_{n-1})A_{l,n-1}.  \]
Since $L_0(x) = x^{\lambda_0}$, we have $A_{l,0} = \frac{1}{\lambda_0 + \lambda_l + 1}$ and thus  
\[ A_{l,n} =  \frac{ \Pi_{j=0}^{n-1} (\lambda_l -\lambda_{j})}{\Pi_{j=0}^{n}(1 + \lambda_l + \lambda_j)}. \] 
\end{proof}

 Now $(A_{l,n})$ can be viewed as an infinite-dimensional lower triangular matrix. Thus, its inverse can 
be calculated by back-substitution. The next lemma gives an explicit formula for the inverse
(in the sense of matrix-multiplication).

\begin{lemma}\label{leminv}
Given the infinite-dimensional matrix from Lemma~\ref{lemone}. 
Then $(A_{l,n})_{l,n}$ has the following inverse 
\[ R_{\alpha,\beta} =
\begin{cases} 
(1 + 2 \lambda_\alpha) \frac{\Pi_{j=0}^{\alpha-1} (1 + \lambda_\beta + \lambda_j)}{\Pi_{j=0,j\not =
\beta}^\alpha (\lambda_\beta - \lambda_j)} & \beta \leq \alpha, \\
0  & \beta > \alpha. \end{cases} \]
\end{lemma} 

\begin{proof}
Since both matrices are lower triangular, a matrix-multiplication involves only finitely many terms. 
For $\beta < \alpha$, 
\[ \sum_{s=\beta}^\alpha R_{\alpha,s} A_{s,\beta} = 
 (1 + 2 \lambda_\alpha)  
  \sum_{s=\beta}^\alpha 
  \frac{\Pi_{j=0}^{\alpha-1} (1 + \lambda_s + \lambda_j)}{\Pi_{j=0,j\not = s}^\alpha (\lambda_s - \lambda_j)}
  \frac{ \Pi_{j=0}^{\beta-1} (\lambda_s -\lambda_{j})}{\Pi_{j=0}^{\beta}(1 + \lambda_s + \lambda_j)} \]
\[ = (1 + 2 \lambda_\alpha)  
  \sum_{s=\beta}^\alpha  
    \frac{\Pi_{j=\beta+1}^{\alpha-1} (1 + \lambda_s + \lambda_j)}{\Pi_{j=\beta,j\not = s}^\alpha (\lambda_s
    - \lambda_j)}.
\]
By substituting $\mu_s = \lambda_{s+\beta}$ we have 
\begin{equation}\label{mamo} \sum_{s=\beta}^\alpha R_{\alpha,s} A_{s,\beta}  =   (1 + 2 \lambda_\alpha)  
  \sum_{s=0}^{\alpha-\beta} 
      \frac{\Pi_{j=1}^{\alpha-\beta-1} (1 + \mu_s + \mu_j)}{\Pi_{j=0,j\not = s}^{\alpha-\beta} (\mu_s - 
      \mu_j)}.
\end{equation}
We now prove that 
\begin{equation}\label{eq} T_{n,\kappa} :=  \sum_{s=0}^{n} 
      \frac{\mu_s^\kappa}{\Pi_{j=0,j\not = s}^{n} (\mu_s - \mu_j)} = 0  \quad \mbox{ for any } \kappa 
      < n. \end{equation}
As in \cite{Borwein}, we make use of contour integrals. 
      Indeed, by the residue theorem, we may express $T_{n,\kappa}$ as 
\[ 2 \pi i T_{n,\kappa}    = \int_\Gamma \frac{x^\kappa}{\Pi_{j=0}^{n} (x -\mu_j)} d x, \]
where $\Gamma$ is a contour in the complex plane 
that encloses all $\mu_j$.  If $\Gamma$ is a circle with large enough
radius $R$, we may estimate 
\[ |2 \pi i T_{n,\kappa}| \leq 2 \pi R  \frac{R^{\kappa}}{|R -\max_{j=0,N} ||\mu_j||^{n+1}} , \]
and this expression tends to $0$ for $\kappa < n$ as $R \to \infty$. Thus \eqref{eq} is shown. 
Expanding the product in the numerator on \eqref{mamo} gives the same  terms as in \eqref{eq}. 
If follows that $RA$ is zero below the diagonal. Above the diagonal, all entries are trivially $0$ 
by the lower triangular structure. Thus $RA$ is a diagonal matrix, which uniquely 
fixes the inverse of $A$ up to a scaling of 
the rows. 
Since the 
diagonal entries of the inverse are the inverse of the diagonal entries of $A$, the diagonal entries 
of $R$ are known and as stated in the lemma. Thus $R$ is the inverse of $A$.  
\end{proof}

We are now ready to solve the moment problems \eqref{momprob}. Note, however, that 
we require $a,b$ to be in a weighted $L^2$-space $\|a\|_{L^2([0,1], r dr)}^2:= \int a(r)^2 r dr$.
Taking this into account, we may rewrite the moment problem as having the coefficients 
\mbox{$\scal a_k(r),r^{2 (i-1) + k}\scar_{L^2([0,1], r dr)}$,} $i = 1,\ldots$, given (and similar for $b$). 
It makes sense to expand $a(r)$ and $b(r)$ into Müntz-Legendre polynomials which are orthogonal 
with respect to the weighted $L^2$-inner product. This can be achieved 
by defining the Müntz-Legendre polynomials based on the monomials $x^{\lambda_i}$, but where 
the coefficients  $c_{k,n}$ are taken as those for the sequence $\lambda_i + \frac{1}{2}$. 
The corresponding Müntz-Legendre polynomials $L_n(r)$ are then  
 orthogonal with respect to the $L^2([0,1], r dr)$-product. Using $\lambda_i = 2(i-1)$, $i = 1,\ldots$,
leads to the following definition:
\begin{definition}\label{def4}
For $k = 0,\ldots$ fixed, we define the polynomials
\begin{equation}
\LM{k}{n}(x):= \sum_{l=0}^{n}  {\LCc{k}{l,n}}  x^{2 l + k}, \quad n = 0,\ldots 
\end{equation} 
with 
\[ {\LCc{k}{l,n}} :=  \frac{\Pi_{j=0}^{n-1} (2l + k + 2j + k + 2)}
{\Pi_{j=0,j\not = l}^{n} (2(l  - j) } = 
\frac{\Pi_{j=0}^{n-1} (l + j + k + 1)}
{\Pi_{j=0,j\not = l}^{n} ((l  - j) }. \]
\end{definition}
As explained above, the coefficients ${\LCc{k}{l,n}}$ are those that are obtained 
for  the Müntz-Legendre polynomials 
based on the sequence $\Lams = (2 i + k+ \frac{1}{2})_{i=0,\ldots}$.

From the orthogonality of the original Müntz-Legendre polynomials it follows  easily that 
\[ \scal \LM{k}{n}(x),\LM{k}{m}(x)\scar_{L^2([0,1],x dx)} = 0 \quad n \not= m \]
and 
\[ \scal \LM{k}{n}(x), \text{span}\{ x^{2 l+k} \,|\, l = 0,\ldots n-1\} \scar_{L^2([0,1],x dx)} = 0,  \]
and  the functions $(\LM{k}{n}(x))_{n=0}^{N}$ span the space with basis the monomials 
$( x^{2 n+k})_{n=0}^N$.

We  now arrive at the inversion formula for \eqref{momprob}. 
\begin{proposition}
Let $k\in \N_0$ be fixed. 
Then $a_k \in L^2([0,1],r dr)$ in \eqref{momprob} is uniquely
determined by $d_i^{c,k}$ and can be found by 
\begin{align}
a_k(r) &= \sum_{n=0}^{\infty}  \LM{k}{n}(r)\cff_n  \label{ex1} \\
\cff_n &= \sum_{l=0}^n 
2 (-1)^{n-l} (2 n + k + 1)  \frac{\Pi_{j=0}^{n-1} (1 + k +  (l + j) )}{
 l ! (n - l)!} 
d_{l +1}^{c,k},  \quad n = 0,\ldots
\end{align}
The same holds for $b$ with  $d_i^{s,k}$ in place of  $d_i^{c,k}$. 
\end{proposition}
\begin{proof}
By M\"untz' theorem and its construction, 
$\LM{k}{n}(r)$ is an orthogonal basis for  $L^2([0,1],r dr)$, thus $a_k$
can be expanded as in \eqref{ex1}.  Fix $k$ and plug the expansion into \eqref{momprob}
to obtain 
\begin{equation}\label{matko}  \sum_{n=0}^{\infty} \cff_n \scal  \LM{k}{n}(r), r^{2(i-1) + k} \scar _{L^2([0,1],r dr)} = 
d_i^{c,k} \qquad i = 1,\ldots. \end{equation}   
We have that the matrix entries
\[ \tilde{A}_{n,i}:= \scal  \LM{k}{n}(r), r^{2(i-1) + k} \scar_{L^2([0,1],r dr)}  \qquad i = 1,\ldots\]
equal $\tilde{A}_{n,i} = A_{n,i-1}$, where $A$ is 
the matrix in \eqref{myA} defined 
for the sequence 
$\Lams = (2i + k +\frac{1}{2})_{i}$, $i = 0,\ldots$. 
Thus, by Lemma~\ref{leminv}, we can invert  \eqref{matko} by 
\[ \cff_n = \sum_{l=0}^n  R_{n,l} d_{l+1}^{c,k}  \]
with $R$ as in Lemma~\ref{leminv} for the sequence $\Lams$, i.e., 
for $l \leq n$
\begin{align*}
R_{n,l} &= (2 + 4 n + 2k) \frac{\Pi_{j=0}^{n-1} (2 + 2 (l + j) + 2 k)}{
\Pi_{j=0,j\not = l}^{n} 2 (l -j)}   \\
&= 2 (2 n + k + 1)  \frac{\Pi_{j=0}^{n-1} (1 + k +  (l + j) )}{
\Pi_{j=0,j\not = l}^{n}  (l -j)} \\
 &= 2 (-1)^{n-l} (2 n + k + 1)  \frac{\Pi_{j=0}^{n-1} (1 + k +  (l + j) )}{
 l ! (n - l)!}. 
\end{align*}
\end{proof}

Now we collect the results into a theorem: 
\begin{theorem}\label{thone}
Let $K^{cc},K^{ss},K^{sc},K^{sc}$ contain the entries 
of  the linearized Dirichlet-to-Neumann map in the trigonometric basis as in Definition~\ref{def2}
with a coefficient $\gamma \in L^{2}(\Omega)$.  Then $\gamma$ can be reconstructed 
from the expression in polar coordinates 
\begin{equation}\label{gammsol} 
\begin{split}\gamma(r,\phi) = \frac{1}{2} \sum_{n=0}^{\infty} 
 \LM{0}{n}(r) \cff_{n,0}  
&+ \sum_{k=1}^\infty  
\sum_{n=0}^\infty  \LM{k}{n}(r) \cos(k \phi) \cff_{n,k} \ \\
&  +\sum_{k=1}^\infty   \sum_{n=0}^\infty 
\LM{k}{n}(r)  
\sin(k \phi) \cgg_{n,k}
\end{split}
\end{equation}
with 
\begin{align*} 
\cff_{n,k} &= \sum_{l=0}^n 
2 (-1)^{n-l} (2 n + k + 1)  \frac{\Pi_{j=0}^{n-1} (1 + k +  (l + j) )}{
 l ! (n - l)!} 
\frac{K_{l+1,l+1+k}^{cc}}{(l+1) (l+1+k) \pi}, \\
\cgg_{n,k} &= \sum_{l=0}^n 
2 (-1)^{n-l} (2 n + k + 1)  \frac{\Pi_{j=0}^{n-1} (1 + k +  (l + j) )}{
 l ! (n - l)!} 
 \frac{K_{l+1,l+1+k}^{cs}}{(l+1) (l+1+k) \pi} . 
\end{align*}
and  $\LM{k}{n}(r)$ defined in Definition~\ref{def4}.
\end{theorem}

Of course, a similar formula holds for $c$ in the Schrödinger case. 
Note that the functions $\LM{k}{n}(r)  \cos(k \phi)$, $n,k =0,\ldots$ and 
$\LM{k}{n}(r)  \sin(n \phi)$, $n = 0,\ldots$, $k = 1,\ldots$ 
provide an orthogonal basis in $L^2(\Omega)$, thus, we can find 
conditions for $\gamma$ being in  $ L^2(\Omega)$ in terms of the 
coefficients $\cff_{n,k},\cgg_{n,k}$.

\begin{corollary}
$K^{cc},K^{ss},K^{sc},K^{sc}$ are the
the entries 
of  the linearized Dirichlet-to-Neumann map in the trigonometric basis with $\gamma \in L^2(\Omega)$
if and only if \eqref{thisy} and 
\eqref{misy} is satisfied and 
\[ \sum_{n=0}^\infty \sum_{k=0}^\infty \frac{1}{2 n +  k + 1} \cff_{n,k}^2 + 
\sum_{n=0}^\infty \sum_{k=1}^\infty \frac{1}{2 n +  k + 1}  \cgg_{n,k}^2  < \infty \]
with $\cff_{n,k}$, $\cgg_{n,k}$ given in Theorem~\ref{thone}.
\end{corollary}
\begin{proof}
According to \cite{Borwein}[Theorem 2.4], we can calculate the 
$L^2([0,1],r dr)$-norm of $\LM{k}{n}(r)$ to $\frac{1}{\sqrt{4 n + 2 k + 2)}}$,
and after an normalization and Parceval' s identity, the result follows. 
\end{proof}

Moreover, we can characterize which parameter 
$\gamma$  are identifiable from finite  data, i.e., when only  trigonometric 
function up to a certain frequency are  used, and we can completely characterize
the finite-data Dirichlet-to-Neumann matrices.
\begin{corollary}\label{cor2}
Given matrices $K^{cc}_{i,j},$ $K^{ss}_{k,l},$ $K^{sc}_{i,k}$, $K^{cs}_{l,j}$, 
where $i,j = 0,\ldots N$ and $k,l = 1,\ldots N$ that satisfy \eqref{thisy} and 
\eqref{misy}. Then there exists a $\gamma \in L^2(\Omega)$ such that 
these matrices contain the coefficients of the linearized Dirichlet-to-Neumann 
map for the impedance tomography problem for the trigonometric basis 
$\{ \cos{l \phi}\} \cup \{\sin(k \phi) \}$ for $l = 0, \ldots N$, $k = 1,\ldots N$. 
The parameter $\gamma$ can be expressed as 
\begin{equation}\label{aaa}
\begin{split} \gamma(r,\phi) = \frac{1}{2} \sum_{n=0}^{N-1} 
 \LM{0}{n}(r) \cff_{n,0}  
&+ 
\sum_{n=0}^{N-1} 
\sum_{k=1}^{N-(n+1)} 
\LM{k}{n}(r)  \cos(k \phi) \cff_{n,k} \ \\
&  + \sum_{n=0}^{N-1}  \sum_{k=1}^{N-(n+1)}  
\LM{k}{n}(r)  
\sin(k \phi) \cgg_{n,k},
\end{split}
\end{equation}
and the coefficients $\cgg_{n,k},$ $\cff_{n,k}$ are uniquely specified by the matrices
$K^{cc},$ $K^{ss}$ $K^{sc}$, $K^{cs}$.
\end{corollary}

\begin{proof}
We may define infinite-dimensional matrices 
$K^{cc}$ $K^{ss}$ $K^{sc}$, $K^{cs}$ by extending the given matrices with 
$0$ for indices that are larger than $N$. These matrices satisfy again 
\eqref{thisy} and \eqref{misy}. The inversion formula 
\eqref{gammsol} gives then a $\gamma$ which induces a Dirichlet-to-Neumann 
map with these matrices. The coefficients in \eqref{aaa} only involve the 
given matrices with indices smaller than $N$.
\end{proof}

\begin{remark}\rm 
Of course, \eqref{aaa} is not the only parameter that solves the impedance tomography
problem with finite data, but we may add higher terms as in \eqref{gammsol} at our 
wish because they are in the nullspace. What components to add is, of course, 
subjective  and may be justified by  a
regularization approach (or a Bayesian perspective) that  chooses a $\gamma$ that 
fits additional needs. 
\end{remark}

\begin{remark}\rm 
Equation~\eqref{aaa} also shows the typical difficulty of impedance tomography:
Since the functions $\LM{k}{n}(r)$ are quite flat close to the center  and similar to each other there,
the resolution in the interior is typically quite bad and becomes worse as we approach 
the center of the disk. By expanding inclusion sets into the Müntz-Legendre basis, 
 it is possible to provide estimates for the resolution limits of inclusions 
in the linearized case with finite and noisy data. 
\end{remark}

\begin{remark}\rm 
In the Schrödinger case, the  situation is slightly different: Given the  first $n\times n$ portion 
of the linearized Dirichlet-to-Neumann map, we can still recover the coefficients 
of $c$ as in \eqref{aaa}. However, the Hankel-part provides additional information, 
namely the moments $\int_0^1 a_{l} r^{l+1} dr$ and  $\int_0^1 b_{l} r^{l+1} dr$
for $l = n+1,\ldots, 2n$, which could be used to get additional information about $c$.
Note, however, that we cannot extend the $J^{cc}$ $J^{ss}$ $J^{sc}$, $J^{cs}$ 
by $0$ as in the proof of Corollary~\ref{cor2} as this would destroy the Hankel structure. 
Thus, this means that  
the first $n\times n$ part of the matrices have to satisfy certain compatibility condition 
with matrix entries of higher frequencies larger than $n$. 
\end{remark}

\begin{remark}\rm
Since coefficients of the form \eqref{aaa} are uniquely specified by the finite-data 
linearized Dirichlet-to-Neumann operator and the inversion formula provides an 
explicit inverse that is 
bounded (though with a terrrible $n$-dependent bound), we may employ the 
inverse function theorem to conclude that also coefficients of the form \eqref{aaa} 
are locally uniquely determined be the (nonlinear) Dirichlet-to-Neumann mapping 
in the finite-data case (that is, involving  $f$, $g$ only up to certain frequencies $n$). 
Here locally means that $\gamma -1$ with $\gamma$ as in \eqref{aaa} must be sufficiently small 
in the $L^\infty$-norm and the bound on the $L^\infty$-norm depends on $n$. 
\end{remark}

\section{Incomplete measurements}
In this section we consider the impedance tomography problem with 
incomplete measurements. That is, we assume only parts of the boundary
accessible for measurements and impose homogeneous Dirichlet condition 
on the rest. Results on uniqueness in this case are scarce, in particular, 
in the two-dimensional case. For the 3D (and nonlinear) case, results 
on unique identifiability can be found in \cite{Krup,Kesouh}.

\subsection{Incomplete measurements on the half disk}
Specifically, we first consider the problem on the upper half disk where the upper half circle 
is accessible to measurements.
\begin{align*}  \Omega_h &= \{(x,y) \in \R^2 \,|\, x^2 + y^2 \leq 1, y > 0  \} \\
\Gamma &= \partial \Omega_h \cap \{ y > 0 \} \quad
\Gamma_0 =  \partial \Omega_h \cap \{ y = 0 \}. 
\end{align*}
We consider the problem 
\begin{equation}\label{incsolu}
\begin{split}
 -\div(\gamma \nabla) u_f &= 0 
 \qquad \qquad 
 \mbox{ in } \Omega 
 \\ 
          u_f & = \begin{cases} f & \mbox{ on } \Gamma, \\ 
            0 & \mbox{on }  \Gamma_0. 
                \end{cases} 
\end{split}
 \end{equation}
As data we consider the Neumann data again on $\Gamma$ 
 which 
introduces the incomplete Dirichlet-to-Neumann operator 
\[ \Lambda_{1,\gamma}^{inc} : f|_{\Gamma} \to \gamma \frac{\partial}{\partial n} u_f |_{\Gamma}. \] 
Here, we assume that $f$ is so that the Dirichlet problem has a solution in $H^1(\Omega_h)$, 
which is the case when the zero-extension of $f$ to $\partial \Omega_h$, 
\begin{equation}\label{def} f^e:= \begin{cases} f & \mbox{ on } \Gamma \\ 
           0 & \mbox{ on }  \Gamma_0 \end{cases}  
\end{equation}           
is in 
$H^{\frac{1}{2}}(\partial \Omega_h)$, which we assume throughout. We denote that 
spaces as $H_0^{\frac{1}{2}}(\partial \Omega_h)$.
Is is not difficult to verify that the incomplete Dirichlet-to-Neumann operator can be written as
\[ \langle \Lambda_{1,\gamma}^{inc}f,g\rangle_{H_0^{-\frac{1}{2}}, H_0^{\frac{1}{2}}} = 
\langle\Lambda_{1,\gamma} f^e,g^e\rangle_{H^{-\frac{1}{2}}, H^{\frac{1}{2}}}.  \]
By this formula, we find for the linearized problem (around $\gamma =1$) the 
following expression
\begin{equation}\label{incdtn} \langle{\Lambda_{1,\gamma}^{inc}}'f,g\rangle_{H_0^{-\frac{1}{2}}, H_0^{-\frac{1}{2}}} 
= \int_{\Omega_h} \gamma(x) \nabla u_{f,1}(x). \nabla u_{g,1}(x) dx,  
\end{equation}
where $u_{f,1}, u_{g,1}$ are solutions to \eqref{incsolu} with $\gamma = 1$, 
(i.e., the Laplace equation) with the respective
Dirichlet boundary conditions.

We prove the following theorem: 
\begin{theorem}\label{onemain}
Let $\gamma \in L^2(\Omega_h)$. Then $\gamma$ is uniquely determined by the 
linearized 
incomplete Dirichlet-to-Neumann map  ${\Lambda_{1,\gamma}^{inc}}'$.
In fact, it suffices to have in \eqref{incdtn} 
$f(\phi) = \sin(n \phi)$ and $g(\phi) = \sin(k \phi)$,  $\phi \in [0,\pi]$ for $n,k = 1,\ldots$
\end{theorem}

\begin{proof}
Setting $f_n =  \sin(n \phi)$  and, by using polar coordinates, we obtain 
the associated solution $u_{f_n,1}(r,\phi) = r^n \sin(n \phi),$ $r \in (0,1)$, 
$\phi \in [0, \pi]$, and similar for $g_k(\phi) = \sin(k \phi)$ with 
 $u_{g_k,1}(r,\phi) = r^k \sin(k \phi),$ $n,k \geq 1$. 
By the Schwarz' reflection principle we may extend $u$ (antisymmetrically with respect 
to $y = 0$) to a harmonic
map on the unit disk and obtain the functions 
\[ u_{\tilde{f}_n,1}(r,\phi) = r^n \sin(n \phi), \quad  
u_{\tilde{g}_k,1}(r,\phi) = r^k \sin(k \phi) , \quad r \in (0,1), \phi \in [0,2 \pi]. \]
Extending $\gamma$ symmetrically to the lower half disk by 
\[ \tilde{\gamma}(x,-y) = \gamma(x,y), \quad y <0, \]
we observe that  the integral in \eqref{incdtn} over the lower half disk 
equals that of the upper half disk. Thus, we can express  the linearized 
incomplete  Dirichlet-to-Neumann map via the linearization of 
the full  Dirichlet-to-Neumann map on the unit disk with $\tilde{\gamma}$. 
That is 
\[  \langle{\Lambda_{1,\gamma}^{inc}}'f_n,g_k\rangle  = 
\frac{1}{2} \langle \Lambda_{1,\tilde{\gamma}}' \tilde{f}_n,  \tilde{g}_k \rangle, \]
where $\tilde{f}_n$ and $\tilde{g}_k$ are the antisymmetric extensions 
to the lower half, i.e., $\tilde{f}_n(\phi) = \sin(n \phi)$, 
$\tilde{g}_k(\phi) = \sin(k \phi)$, for $\phi \in [0, 2\pi]$. 
Finally, we may expand $\gamma$ on the upper half disk into a 
pure cosine series, i.e., using polar coordinates
\[ \gamma(r,\phi) = \sum_{n=0}^\infty a_n(r) \cos(n \phi), \quad \phi \in [0,\pi],\]
which, by our symmetric extension immediately gives the expansion of 
$\tilde{\gamma}$ 
\[ \tilde{\gamma}(r,\phi) = \sum_{n=0}^\infty a_n(r) \cos(n \phi), \quad \phi \in [0,2 \pi].\]
Now we may use the formula for $\Lambda_{1,\tilde{\gamma}}'$ to conclude 
\[  2({\Lambda_{1,\gamma}^{inc}}'f_n,g_k)  = 
n k \pi \myg_{n,k} \int_0^1 r^{n+k-1} a_{|n-k|}(r) dr.  \]
Thus, by considering 
$({\Lambda_{1,\gamma}^{inc}}'f_n,g_{n+l})$, for $n = 1,\ldots$, $l = 0,\ldots$,
and the completeness of the function families  $(r^{2n+l-1})_{n=1}^\infty$ for 
$l = 0,\ldots$ that follows from M\"untz' theorem, we have shown that 
the coefficients $a_n$ are uniquely determined, which completes the proof. 
\end{proof}

An inversion formula is provided by \eqref{gammsol} using only a cosine expansion, 
\[ \gamma(r,\phi) = \frac{1}{2} \LM{0}{n}(r)  \cff_{n,0}  
+ \sum_{k=1}^\infty  
\sum_{n=0}^\infty  \LM{k}{n}(r) \cos(k \phi)  \cff_{n,k},  \]
and noting that with our sine boundary functions 
\[ f_n(\phi) = \sin(n \phi) \quad   g_k(\phi) = \sin(k \phi),\] 
we have 
$2({\Lambda_{1,\gamma}^{inc}}'f_{i},g_{j}) = K_{i,j}^{ss}$ 
leading to the coefficient formula 
\begin{align} \label{formulahalf}
\cff_{n,k} &= \sum_{l=0}^n 
2 (-1)^{n-l} (2 n + k + 1)  \frac{\Pi_{j=0}^{n-1} (1 + k +  (l + j) )}{
 l ! (n - l)!} 
\frac{2\langle{\Lambda_{1,\gamma}^{inc}}'f_{l+1},g_{l+k+1}\rangle}{(l+1) (l+1+k) \pi}. 
\end{align}

\subsection{Incomplete measurements on the disk}
We now study a similar problem as before but on the full disk and 
where the measurements are available only on a interval on the boundary. 
Specifically, we assume access only to a interval on the unit circle on the upper half of the form 
\begin{equation}\label{yyy}
\phi \in I: =  [\frac{\pi}{2}-\alpha,  \frac{\pi}{2}+ \alpha] ,
\end{equation} where, $0 < \alpha < \frac{\pi}{2}$ and 
$\phi$ is the angular coordinate. The restriction 
to  $\alpha < \frac{\pi}{2}$ is only for convenience and could be dropped. 
Again considering the linearized case, we study boundary value problems 
with given $f$ defined on $I$. 
\begin{equation}\label{xx} 
\begin{split} 
\Delta u_{f,1}  &= 0 \qquad \mbox{in } \Omega = \{(x,y) \in \R^2\,|\, x^2 +y<^2 <1 \}, \\
         u_{f,1} & = \begin{cases} f & \mbox{ on } I \cap \partial \Omega,  \\ 
            0 & \mbox{ on }   \partial \Omega \setminus I.   
                \end{cases} 
\end{split} 
\end{equation}
As above, we have to restrict ourselves to functions $f$ for which the zero extension to 
$\partial \Omega$ is in $H^\frac{1}{2}$, or, equivalently, to Dirichlet values where 
the problem \eqref{xx}
has a solution in $H^1(\Omega)$. 

We study the analogous linearized problem  with data available only on $I$, 
that is, we define the linearized incomplete Dirichlet-to-Neumann operator 
\begin{equation}\label{zzz}   \langle{\Lambda_{1,\gamma}^{inc}}'f_n,g_k\rangle := 
\int_{\Omega} \gamma \nabla  u_{f,1}. \nabla  u_{g,1} dx \qquad f^e,g^e \in H^{\frac{1}{2}}. 
\end{equation} 
Here $f^e,g^e$ denotes the zero extension to the whole boundary as before. 

We proof the following theorem:
\begin{theorem}\label{myth}
Let $\gamma \in L^2(\Omega)$ be such that $\gamma = 0$ 
in a neighborhood of the endpoints of $I$. 
 Let the linearized 
incomplete Dirichlet-to-Neumann operator ${\Lambda_{1,\gamma}^{inc}}'$ be defined 
in \eqref{zzz} with data on the interval $I$. 
Then $\gamma$ is uniquely determined by the values 
\[ \left\{ \langle{\Lambda_{1,\gamma}^{inc}}'f ,g \rangle , \, |\, f^e ,g^e \in  H^{\frac{1}{2}} \right\} . \] 
\end{theorem}

The proof is based on conformal mappings; more specifically on the following lemma:
\begin{lemma}\label{lembef}
There exists a conformal map $\con$ from the  the upper half  of the unit disk to the full 
unit disk 
which can be extended to a homeomorphism from and to  the closures of the respective sets 
such that the left and right endpoints of the half disk, $(-1,0)$ and $(1,0)$ are mapped 
to the left and right endpoints of the interval in \eqref{yyy}.
\end{lemma}

\begin{proof}
A conformal map that takes the upper half disk to the unit disk is well known:
The combination of the mappings $\f_0 = \frac{1+z}{1-z}$ and $z \to z^2$ 
gives the map 
$\f_1(z) = (1 + z)^2/(1 -z)^2$ \cite[Ex 2, 3, pp.~210] {Stein} which 
takes the upper half disk to the upper half plane, while the fractional 
transform $\f_2(z):= \frac{z-i}{z+i}$ maps
the upper half plane to the unit disk. Combining these two maps gives 
\[ \f(z) = \f_2 \circ \f_1 
= \frac{{(1+z)^2} - i (1-z)^2}{{(1+z)^2} +i(1-z)^2},
\]
which maps the upper half disk to the unit disk and leaves the half-disk's endpoints
at $z = \pm 1 + 0 i$ invariant. Moreover,  the boundary $\{ x^2 +y^2 = 1,y \geq 0\}$ is mapped 
to itself, while the lower part of the boundary, $\Omega_h\cap \{y = 0\}$, is mapped to the lower half of the
unit circle.  
By Caratheodory's theorem the mappings extends to homeomorphisms, but 
for this example, we can calculate explicitly that the mappings on the boundary 
are invertible. Finally, we observe that the mapping is smooth on the closure of the upper half disk.  
To achieve a mapping with the specifications about the endpoints, we consider the 
conformal automorphisms of the unit disc, 
\[ \h(z) = e^{i \mu} \frac{z - w}{\overline{w} z -1} \quad \mu \in [0,2 \pi], |w| < 1. \] 
The conditions $\f(1) = e^{i (\frac{\pi}{2} -\alpha)}$ and 
$\f(-1) = e^{i (\frac{\pi}{2} +\alpha)}$ can be satisfied by  $\mu = \pi$
and $w = -i \frac{\cos(\alpha)}{1 + \sin(\alpha)}$. It is easily verified that 
$|w| < 1$ holds such that $\h$ is indeed an automorphism of the unit circle. 
Setting $\con = \h \circ \f$ provides the desired map. 

We note that the inverse of the conformal map is given by 
\[   \con^{-1} =  \f_0^{-1} \circ  \sqrt{} \circ \f_2^{-1} \circ \h^{-1}, \]
where $\f_0^{-1} = \frac{z-1}{z+1}$, $f_2^{-1} =  i \frac{z+1}{1-z}$ and 
$\sqrt{}$ is the complex square root with branch cut at the negative imaginary axis. 
By calculating the derivative of $\con$, we observe that $|D \con|$ vanishes 
only at the end points of the circle $(1,0)$ and $(-1,0)$.  Since they are 
mapped to the endpoints of $I$,  the inverse 
is smooth on $\overline{\Omega}$ away from these points. 
\end{proof}

\begin{proof}[Proof of Theorem~\ref{myth}]
Start with the functions $\tilde{f}^e$ and $\tilde{g}^e$ that are defined 
on the boundary of the upper half circle by \eqref{def} with $f = \sin(n \phi)$ and 
$g = \sin(k \phi)$ on the upper boundary $\Gamma$ (and extended by zero to $\Gamma_0$).
Consider their images by the conformal map of Lemma~\ref{lembef}:
\[ f^e = \tilde{f}\circ \con \quad g^e = \tilde{g}\circ \con.  \]
These are continuous functions on the unit circle and they are supported in 
the interval $I$ by construction. Moreover, since 
 $\tilde{f}^e$, $\tilde{g}^e$ are the boundary values of the 
 harmonic functions $\tilde{u}_{\tilde{f}} (r,\phi) = r^n \sin(n \phi) = \mbox{Im}(z^n) $ and 
 $r^k \sin(k \phi) = \mbox{Im}(z^k)$, the so defined functions $f^e,g^e$ correspond 
 (by conformality) to harmonic functions 
 \[ u_{f,1} = \mbox{Im}(({\phi^{-1}})^n) \quad u_{g,1} = \mbox{Im}(({\phi^{-1})}^k). \] 
We have by conformal invariance of the Dirichlet integral that 
\[ \int_{\Omega} |\nabla u_{f,1}|^2 dx =  \int_{\Omega_h} |\nabla \tilde{u}_{0,\tilde{f}}|^2 dx \]
and 
\[ \int_{\Omega} |u_{f,1}|^2 dx =  \int_{\Omega_h} |\tilde{u}_{0,\tilde{f}}|^2 |D \con|(x) dx 
\leq   \sup_{x \in \overline{\Omega_h}} |D \con(x)| \int_{\Omega_h} |\tilde{u}_{0,\tilde{f}}|^2 
< \infty \]
because  $|D \con(x)|$ is smooth.  Thus, $u_{f,1}$ (and clearly similarly  $u_{g,1}$)  
are in $H^1$, thus, 
$f^e$, and $g^e$ are in $H^{1/2}$ and the restriction to $I$ can 
be used as data for $({\Lambda_{1,\gamma}^{inc}}'f ,g )$.
We then have 
\begin{align*}  \langle{\Lambda_{1,\gamma}^{inc}}'f ,g \rangle = 
\int_{\Omega} \gamma \nabla u_{f,1}. \nabla u_{g,1}dx = 
\int_{\Omega^h} \gamma(\con(x))   \nabla \tilde{u}_{f,1}. \nabla \tilde{u}_{g,1}dx, 
\end{align*}
where $\tilde{u}_{f,1}(r,\phi) = r^n \sin(n\phi)$ and  $\tilde{u}_{g,1}=  r^k \sin(k\phi)$. 
We verify that $\gamma(\con^{-1})$ is in $L^2(\Omega_h)$:
\[ \int_{\Omega^h}  \gamma(\con(x))^2 dx = 
\int_{\Omega}  \gamma(x)^2 dx |D \con^{-1}| dx \leq 
\sup_{x \in {\rm supp}(\gamma)} |D \con^{-1}|(x)|  \int_{\Omega}  \gamma(x)^2 dx  < \infty. \] 
The later inequality holds because $\con^{-1}$ is  smooth away from the interval endpoints. 
Thus $ \gamma\circ \con $ is uniquely determined by Theorem~\ref{onemain} and, 
hence, so is $\gamma$. 
\end{proof}

The inversion formula in this case is 
\[ \gamma(r,\phi) =\left[ \frac{1}{2} \LM{0}{n}(r) \cff_{n,0}  
+ \sum_{k=1}^\infty  
\sum_{n=0}^\infty  \LM{k}{n}(r) \cos(k \phi) \cff_{n,k}  \right] \circ \con^{-1} \]
with coefficients given by \eqref{formulahalf}, where $f_n,g_n$ are the 
transformed sine functions as in the proof. 

Also in this situation it is possible to study the finite-data case and again derive a formula as 
in Corollary~\ref{cor2}. The ``flatness'' of the Müntz-Legendre polynomials is transformed 
by the conformal map to the region opposite of the data interval $I$, which shows the 
expectable fact that  variations in the parameter $\gamma$ that are located 
opposite to the data site are hardest to reconstruct. 

\section{Final comments}
We have shown that the expansion of the parameters 
into  trigonometric/Müntz-Legendre polynomials is an interesting tool for the
linearized impedance 
tomography problem as it allows for an explicit inversion formula and a transparent
characterizations of what can be identified in the finite data case. 

When it comes to numerical calculations, even with an explicit 
inversion formula, problems may occur because of, e.g., rounding errors. Note that the formula
shows the typical features of inversion in the case of ill-posed problems, namely values 
of opposite signs have to be added, which may leads to cancellation. Thus, it is 
a good idea to include a regularization also here.  
On the other hand, the inversion formula certainly provides a fast solution method 
compared to a PDE-based approach 
as it operates directly on
the data space and no  interior grid has to be used. 

It would be interesting to analyze the corresponding three-dimensional case. There, the 
trigonometric expansion in the angular coordinate is naturally replaced by spherical 
harmonics. When also expanding $\gamma$ in this basis, we similarly come to integrals 
that involve triple combinations of spherical harmonics, which leads to quite complicated 
combinatorial coefficients. We do not know  whether an approach for an inversion 
formula succeeds in this case. 

\bibliographystyle{siam}
\bibliography{LEIT} 

\begin{thebibliography}{10}

\bibitem{AsPa}
{\sc K.~Astala and L.~P\"aiv\"arinta}, {\em Calder\'on's inverse conductivity
  problem in the plane}, Ann. of Math. (2), 163 (2006), pp.~265--299.

\bibitem{Borcea}
{\sc L.~Borcea}, {\em Electrical impedance tomography}, Inverse Problems, 18
  (2002), pp.~R99--R136.

\bibitem{Borwein}
{\sc P.~Borwein, T.~Erd\'elyi, and J.~Zhang}, {\em M\"untz systems and
  orthogonal {M}\"untz-{L}egendre polynomials}, Trans. Amer. Math. Soc., 342
  (1994), pp.~523--542.

\bibitem{Cal}
{\sc A.-P. Calder\'on}, {\em On an inverse boundary value problem}, in Seminar
  on {N}umerical {A}nalysis and its {A}pplications to {C}ontinuum {P}hysics
  ({R}io de {J}aneiro, 1980), Soc. Brasil. Mat., Rio de Janeiro, 1980,
  pp.~65--73.

\bibitem{Is}
{\sc V.~Isakov}, {\em Inverse problems for partial differential equations},
  vol.~127 of Applied Mathematical Sciences, Springer-Verlag, New York, 1998.

\bibitem{Kesouh}
{\sc C.~E. Kenig, J.~Sj\"ostrand, and G.~Uhlmann}, {\em The {C}alder\'on
  problem with partial data}, Ann. of Math. (2), 165 (2007), pp.~567--591.

\bibitem{Krup}
{\sc K.~Krupchyk and G.~Uhlmann}, {\em The {C}alder\'on problem with partial
  data for conductivities with 3/2 derivatives}, Comm. Math. Phys., 348 (2016),
  pp.~185--219.

\bibitem{muntz}
{\sc C.~H. {M\"untz}}, {\em {\"Uber den Approximationssatz von {\it
  Weierstra\ss}}}, in Mathematische Abhandlungen Hermann Amandus Schwarz,
  C.~Caratheodory, G.~Hessenberg, E.~Landau, and L.~Lichtenstein, eds.,
  Springer Berlin, 1914, pp.~303--312.

\bibitem{Nach}
{\sc A.~I. Nachman}, {\em Global uniqueness for a two-dimensional inverse
  boundary value problem}, Ann. of Math. (2), 143 (1996), pp.~71--96.

\bibitem{Stein}
{\sc E.~M. Stein and R.~Shakarchi}, {\em Complex analysis}, vol.~2 of Princeton
  Lectures in Analysis, Princeton University Press, Princeton, NJ, 2003.

\bibitem{SylUl2}
{\sc J.~Sylvester and G.~Uhlmann}, {\em A global uniqueness theorem for an
  inverse boundary value problem}, Annals of Mathematics, 125, pp.~153--169.

\bibitem{SylUl1}
\leavevmode\vrule height 2pt depth -1.6pt width 23pt, {\em A uniqueness theorem
  for an inverse boundary value problem in electrical prospection}, Comm. Pure
  Appl. Math., 39, pp.~91--112.

\bibitem{Uhlmanrev}
{\sc G.~Uhlmann}, {\em Electrical impedance tomography and {C}alder\'on's
  problem}, Inverse Problems, 25 (2009), pp.~123011, 39.

\end{thebibliography}
\end{document}